\documentclass[a4paper,11pt]{amsart}
\usepackage{amsmath, amsthm, amssymb}
\usepackage{graphicx}
\usepackage{mathbbol}
\usepackage{txfonts}
\usepackage[usenames]{color}

\newtheorem{thm}{Theorem}[section]
\newcommand{\bt}{\begin{thm}}
\newcommand{\et}{\end{thm}}

\newtheorem{cor}[thm]{Corollary}   

\newcommand{\bc}{\begin{cor}}
\newcommand{\ec}{\end{cor}}

\newtheorem{lem}[thm]{Lemma}   

\newcommand{\bl}{\begin{lem}}
\newcommand{\el}{\end{lem}}

\newtheorem{prop}[thm]{Proposition}
\newcommand{\bp}{\begin{prop}}
\newcommand{\ep}{\end{prop}}

\newtheorem{defn}[thm]{Definition}

\newcommand{\bd}{\begin{defn}}    
\newcommand{\ed}{\end{defn}}

\newtheorem{rmrk}[thm]{Remark}   

\newcommand{\br}{\begin{rmrk}}
\newcommand{\er}{\end{rmrk}}

\newcommand{\Ric}{\operatorname{Ric}}

\newcommand{\be}{\begin{equation}}
\newcommand{\ee}{\end{equation}}

\newcommand{\R}{\mathbb{R}}

\newcommand{\diam}{\operatorname{diam}}

\newcommand{\disjointunion}{\sqcup}

\newcommand{\vol}{\operatorname{Vol}}

\DeclareMathOperator{\Rm}{Rm}

\newcommand{\sphere}{\mathbb{S}}

\begin{document}

\title[Continuity of Ricci Flow through Neckpinch Singularities ]{Continuity of Ricci Flow through Neckpinch Singularities}

\author{Sajjad Lakzian}
\thanks{Lakzian is partially supported as a doctoral student
by NSF DMS \#1006059.}
\address{CUNY Graduate Center and Lehman College}
\email{SLakzian@gc.cuny.edu}

\date{}

\keywords{}

\begin{abstract}
	In this article, we consider the Angenent-Caputo-Knopf's Ricci Flow through neckpinch singularities. We will explain how one can see the A-C-K's Ricci flow through a neckpinch singularity as a flow of integral current spaces. We then prove the continuity of this weak flow with respect to the Sormani-Wenger Intrinsic Flat (SWIF) distance.
\end{abstract}

\maketitle


\section{Introduction}

There are many parallels between Hamilton's Ricci Flow and Mean Curvature Flow. While Ricci Flow with surgery was proposed by Hamilton (see ~\cite{Hamilton-Singularities}) and modified and developed by Perelman (see \cite{Perelman-RFS} and ~\cite{Perelman-extinction}), a \emph{canonical} Ricci Flow through singularities which could be a gateway to a notion of Weak Ricci Flow is only being explored and defined recently (see ~\cite{ACK}). On the contrast, weak MCF was developed by Brakke by applying Geometric Measure Theory and viewing manifolds as varifolds. Recenly White proved that Brakke Flow is continuous with respect to the Flat distance, when the varifolds are viewed as integral current spaces (see ~\cite{White-2009}.) For the Ricci Flow - in contrast with MCF - there is no apriori ambient metric space so one needs to work with intrinsic notions of convergence. The theory of Intrinsic Flat distance has been recently developed by Sormani-Wenger in ~\cite{SorWen2} which provides a framework for our work. Here, we prove that the weak Ricci Flow through neckpinch singularity proposed by Angenent-Caputo-Knopf is continuous with respect to Sormani-Wenger Intrinsic Flat distance when the Riemannian manifolds flowing through the neckpinch singularity and the resulting singular spaces are viewed as integral current spaces. 

Consider the Ricci Flow on the $\sphere^{n+1}$ starting from a rotationally symmetric metric $g_0$. Angenent-Knopf in \cite{Neckpinch} showed that if $g_0$ is pinched enough, then the flow will develop a neckpinch singularity (see Definition \ref{defn-neckpinch}) in finite time $T$ and they computed the precise asymptotics of the profile of the solution near the singular hypersurface and as $t \nearrow T$. 

Later, in ~\cite{ACK}, Angenent-Caputo-Knopf proved that one can define a smooth forward evolution of Ricci Flow through the neckpinch singularity. They achieved that basically by taking a limit of Ricci Flows with surgery and hence showed that Perelman's conjecture that a \emph{canonical} Ricci Flow with surgery exists is actually true in the case of the sphere neckpinch. Since the smooth forward evolution performs a surgery at the singular time $T=0$ and on scale $0$, therefore at all positive times the flow will consists of two disjoint smooth Ricci flows on a pair of manifolds.

In order to define a weak Ricci flow, we must view the pair of manifolds $M_1$ and $M_2$ as a single integral current space.  Recall that an integral current space $(X,d,T)$ defined in \cite{SorWen2} is a metric space $(X,d)$  endowed with an integral current structure $T$ using the Ambrosio-Kirchheim notion of an integral current \cite{AK} so that $X$ is the set of positive density of $T$. Following a suggestion of Knopf, we endow $M = M_1 \disjointunion M_2$ with a metric restricted from a metric space obtained by gluing the manifolds at either end of a thread of length $L(t)$.  The resulting integral current space does not include the thread (nor the point of singularity at time $t=T$) because every point in an integral current space has positive density. We will consider this approach and prove the following continuity result:

\begin{thm}\label{thm-main}
	Let $\left(X(t) , D(t) , T(t) \right)$ be a smooth rotationally and reflection symmetric Ricci flow on $\sphere^{n+1}$ for $t \in (-\epsilon , 0 )$ developing a neckpinch singularity at $T=0$ and continuing for $t \in (0,\epsilon)$ as a disjoint pair of manifolds joined by a thread of length $L(t)>0$ with $L(0)=0$ undergoing Ricci flow as in \cite{ACK}. Then, this is continuous in time with respect to the SWIF distance. 
\end{thm}

Notice that the assumption $T=0$ in Theorem \ref{thm-main} is only for the sake of simplicity . The reflection symmetry in Theorem \ref{thm-main} is there to guarantee the finite diameter at the singular time. In general, we get the following corollary:

\begin{cor}\label{cor-main}
Let $\left(X(t) , D(t) , T(t) \right)$ be a smooth rotationally symmetric Ricci flow on $\sphere^{n+1}$ for , $t \in (-\epsilon , 0 )$ developing a neckpinch singularity at $T=0$ with finite diameter and continuing for $t \in (0,\epsilon)$ as a disjoint pair of manifolds joined by a thread of length $L(t)>0$ with $L(0)=0$ undergoing Ricci flow as in \cite{ACK}. Then, $X(t)$ is continuous in time with respect to the SWIF distance.
\end{cor}

In order to prove Theorem \ref{thm-main}, in Theorem \ref{thm-sub-diffeo-new}, we adapt a result from the previous work of the author with Sormani ~\cite{Lakzian-Sormani} in order to estimate the SWIF distance between our spaces. In Lemmas \ref{lem-smooth}, \ref{lem-pre-surgery} and \ref{lem-post-surgery}, we prove the continuity of the flow prior, at and post the singular time respectively. 

The paper is organized as follows: In Section~\ref{sect-swif}, we give a brief review of the notion of SWIF distance and provide some results from the previous work of the Author with Sormani ~\cite{Lakzian-Sormani} which aid us in estimating the SWIF distance. Section \ref{sect-RF} provides a review of some of the basic fact about Ricci flow as well as a review of the recent work on the Ricci Flow neckpinch and smooth forward evolution (~\cite{Neckpinch}~\cite{AK-precise}~\cite{ACK}) which is key for our work in this paper. Section \ref{sect-main} is devoted to proving Theorem \ref{thm-main} which is the main result of this paper.

\subsection*{Acknowledgments}
The author would like to express his deep gratitude towards Professor Dan Knopf for his hospitality, his time, his effort to explain the details of his work with Angenent and Caputo to the author and for all the helpful discussions the author has had during his visit to University of Texas at Austin and afterward. The author would also like to thank his doctoral adviser, Professor Christina Sormani for her constant support and extremely helpful suggestions regarding both the materials presented here and the structure and exposition of the paper.


\section{Review of Estimates on SWIF Distance}\label{sect-swif}

The notion of an integral current space and the intrinsic flat distance were first introduced by Sormani-Wenger in ~\cite{SorWen2}, using Ambrosio-Kirchheim's notion of an integral current on a metric space ~\cite{AK}.

Recall that if $Z$ is a metric space and  $T_i \in \mathbf{I}_{m}(Z)$ , $i=1,2$ are two $m$-integral currents on $Z$ (see ~\cite{AK} for the definition of current structure for metric measure spaces), the flat distance between $T_i$'s is defined as follows:
\be
	d^Z_{\mathbf{F}}(T_1 , T_2) := \inf \left\{ \mathbf{M}(U) + \mathbf{M}(V) : T_1 - T_2 = U + \partial V \right\}.
\ee

Now let 
\be
\left( X_i , d_i , T_i \right) \;\;  i = 1,2, 
\ee
be two $m$-integral current spaces. Their Sormani-Wenger Intrinsic Flat distance is defined as
\be
	d_{\mathcal{F}}(X_1 , X_2) := \inf d^Z_{\mathbf{F}} \left( \varphi_{1\#} T_1 , \varphi_{2\#} T_2\right)
\ee
where the infimum is taken over all metric spaces $Z$ and distance preserving embeddings $\varphi_i : \bar{X}_i \to Z $. Note that $\bar{X}_i$ are metric completions of $X_i$ and $\varphi_\# T$ is the push forward of $T$. 

We will briefly review some of these notions in Section \ref{subsec-ics-and-mc}.


\subsection{Integral Current Spaces and Metric Completion}\label{subsec-ics-and-mc}

An integral current space $(X , d , T)$ is a metric space $(X , d)$ equipped with a current structure $T$ such that $\operatorname{set}(T) = X$ (see Definition \ref{defn-settled-completion}). 

An oriented Riemannian manifold $\left(M^m, g \right)$ of finite volume can naturally be viewed as an $m$ - integral current space by specifying the current structure $T$ to be the integration of differential $m$ - forms against $M$ 
\be\label{eq-riem-curr}
	T(\omega) = \int_{M^m} \; \omega.
\ee

More generally, an $m$ - integral current structure $T$ of the current space  $(X,d,T)$ is an integral current $T \in \mathbf{I}_m (\bar{X})$ as is defined by Ambrosio-Kirchheims in ~\cite{AK}. The current structure $T$ provides an orientation and a measure on the space which is called the \textbf{mass} measure of $T$ and is denoted by $||T||$. The mass measure of an oriented Riemannian manifold $M^m$ considered as an $m$ - integral current space is just the Lebesgue measure on $M^m$ as can be seen from \ref{eq-riem-curr}. 

The settled completion $\operatorname{set}(X)$ of an integral current space is the set of points in the metric completion $\bar{X}$ with the positive lower density for the mass measure $||T||$ (see Definition \ref{defn-settled-completion}.)

\begin{defn}[~\cite{SorWen2}] \label{defn-settled-completion} Let $\left(X  , d , \mu \right)$ be a metric measure space. The settled completion $X'$ of $X$ is the set of all points $p$ in the metric completion  $\bar{X}$ of $X$ with positive lower density
\be
	\Theta_* (p) = \liminf_{r \to 0} \frac{\mu \left(B(p,r) \right)}{r^m} >0.
\ee

\end{defn}


\subsection{Estimate on SWIF}

Here, we will present results regarding the estimates on the intrinsic flat distance between manifolds which are Lipschitz close or have precompact regions with that property. There are estimates on the SWIF distance given in ~\cite{SorWen2} which employ techniques from geometric measure theory. The results presented in this section uses the simple idea of hemispherical embedding which is easier to understand especially for the reader without a background in GMT.

To estimate the intrinsic flat distance between two
oriented Riemannian manifolds, one needs to find
distance preserving embeddings, $\varphi_i: M^m_i \to Z$, into a common complete
metric space, $Z$.  Since in the definition of the intrinsic flat distance using a Stoke's type formula, one needs to find a 
filling submanifold, $B^{m+1}\subset Z,$ and an
excess boundary submanifold, $A^m\subset Z$, such that
\be\label{Stokes}
\int_{\varphi_1(M_1)}\omega -\int_{\varphi_2(M_2)}\omega=\int_B d\omega +\int_A \omega,
\ee
then, the intrinsic flat distance is bounded above by
\be \label{est-int-flat}
d_{\mathcal{F}}(M^m_1, M^m_2) \le \vol_m(A^m)+\vol_{m+1}(B^{m+1}). 
\ee

Generally, the filling manifold and excess boundary can have corners or
more than one connected component. Below we present results regarding the construction of these manifolds.

\begin{defn}[~\cite{Lakzian-Sormani}]
Let $D>0$ and $M, M'$ are geodesic metric spaces.  
We say that $\varphi: M \to M'$ is a 
\textbf{$D$-geodesic
embedding}
if for any smooth minimal geodesic,
$\gamma:[0,1]\to M$, of
length $\le D$ we have
\be\label{star-0}
d_{M'}(\varphi(\gamma(0)), \varphi(\gamma(1)))=L(\gamma).
\ee
\end{defn}

\begin{prop} [~\cite{Lakzian-Sormani}]\label{prop-hem}
Given a manifold $M$ with Riemannian metrics
$g_1$ and $g_2$ and $D_1, D_2, t_1, t_2>0$.  Let
$M'=M\times[t_1,t_2]$ and let $\varphi_i: M_i\to M'$
be defined by $\varphi_i(p)=(p,t_i)$.   If a metric
$g'$ on $M'$ satisfies
\be \label{ineq-prop-hem}
g' \ge  dt^2 +  \cos^2((t-t_i)\pi/D_i) g_i  \textrm{ for } |t-t_i|<D_i/2
\ee
and
\be \label{eq-prop-hem}
g'=dt^2 +g_i \textrm{ on }M\times \left\{t_i\right\}\subset M'
\ee
then any geodesic, $\gamma:[0,1]\to M_i$,
of length $\le D_i$ satisfies (\ref{star-0}).   If, the diameter is bounded,
$\diam_{g_i}(M)\le D_i$, then $\varphi_i$ is a distance preserving
embedding.

Furthermore, for $q_1, q_2\in M$, we have
\be\label{extra-edges}
d_{M'}(\varphi_1(q_1), \varphi_2(q_2))\ge
d_{M_i}(q_1,q_2).
\ee

\end{prop}

\begin{proof}
	See ~\cite[Proposition 4.2]{Lakzian-Sormani}.
\end{proof}

\begin{prop}[~\cite{Lakzian-Sormani}] \label{prop-squeeze-Z}
Suppose $M_1=(M,g_1)$ and $M_2=(M,g_2)$ are diffeomorphic
oriented precompact
Riemannian manifolds and suppose there exists $\epsilon>0$ such that
\be
g_1(V,V) < (1+\epsilon)^2 g_2(V,V) \textrm{ and }
g_2(V,V) < (1+\epsilon)^2 g_1(V,V) \qquad \forall \, V \in TM.
\ee
Then for any
\be
a_1> \frac{\arccos(1+\epsilon)^{-1} }{\pi}\diam(M_2)
\ee
and
\be
a_2> \frac{\arccos(1+\epsilon)^{-1}}{\pi} \diam(M_1),
\ee
there is a pair of distance preserving embeddings
$\varphi_i:M_i \to M'=\bar{M} \times [t_1, t_2]$ with a metric
as in Proposition~\ref{prop-hem}
where $t_2-t_1\ge \max\left\{a_1, a_2\right\}$.

In fact, the metric $g'$
on $M'$ can be chosen so that
\be \label{star-1}
g'(V,V) \le dt^2(V,V) +  g_1(V,V)+ g_2(V,V) \qquad \forall \, V \in TM'.
\ee

Thus the Gromov-Hausdorff distance
between the metric completions is bounded,
\be \label{sq-Z-GH}
d_{GH}(\bar{M}_1, \bar{M}_2 ) \le a:=\max\left\{a_1, a_2\right\},
\ee
and the intrinsic flat distance
between the settled completions are bounded,
\be\label{77}
d_{\mathcal{F}}(M'_1, M'_2) \le a\left(V_1+ V_2 + A_1+A_2\right),
\ee
\end{prop}

\begin{proof}
	See ~\cite[Lemma 4.5]{Lakzian-Sormani}.
\end{proof}

\begin{thm} [~\cite{Lakzian-Sormani}]\label{thm-sub-diffeo}
Suppose $M_1=(M,g_1)$ and $M_2=(M,g_2)$ are oriented
precompact Riemannian manifolds
with diffeomorphic subregions $U_i \subset M_i$ and
diffeomorphisms $\psi_i: U \to U_i$ such that
\be \label{thm-sub-diffeo-1}
\psi_1^*g_1(V,V)
< (1+\epsilon)^2 \psi_2^*g_2(V,V) \qquad \forall \, V \in TU
\ee
and
\be \label{thm-sub-diffeo-2}
\psi_2^*g_2(V,V) <
(1+\epsilon)^2 \psi_1^*g_1(V,V) \qquad \forall \, V \in TU.
\ee
Taking the extrinsic diameter,
\be \label{DU}
D_{U_i}= \sup\{\diam_{M_i}(W): \, W\textrm{ is a connected component of } U_i\} \le \diam(M_i),
\ee
we define a hemispherical width,
\be \label{thm-sub-diffeo-3}
a>\frac{\arccos(1+\epsilon)^{-1} }{\pi}\max\{D_{U_1}, D_{U_2}\}.
\ee
Taking the difference in distances with respect to the outside manifolds,
\be \label{lambda}
\lambda=\sup_{x,y \in U}
|d_{M_1}(\psi_1(x),\psi_1(y))-d_{M_2}(\psi_2(x),\psi_2(y))|,
\ee
we define heightS,
\be \label{thm-sub-diffeo-4}
h =\sqrt{\lambda ( \max\{D_{U_1},D_{U_2}\} +\lambda/4 )\,}
\ee
and
\be \label{thm-sub-diffeo-5}
\bar{h}= \max\{h,  \sqrt{\epsilon^2 + 2\epsilon} \; D_{U_1}, \sqrt{\epsilon^2 + 2\epsilon} \; D_{U_2} \}.
\ee
Then the Gromov-Hausdorff distance between the metric
completions is bounded,
\be \label{thm-sub-diffeo-6}
d_{GH}(\bar{M}_1, \bar{M}_2 ) \le a + 2\bar{h} +
\max\left\{ d^{M_1}_H(U_1, M_1), d^{M_2}_H(U_2, M_2)\right\}
\ee
and the Intrinsic Flat distance between the
settled completions is bounded,
\begin{eqnarray*}
d_{\mathcal{F}}(M'_1, M'_2) &\le&
\left(\bar{h} + a\right) \left(
\vol_m(U_{1})+\vol_m(U_2)+\vol_{m-1}(\partial U_{1})+\vol_{m-1}(\partial U_{2})\right)\\
&&+\vol_m(M_1\setminus U_1)+\vol_m(M_2\setminus U_2).
\end {eqnarray*}
\end{thm}

\begin{proof}
	See ~\cite[Theorem 4.6]{Lakzian-Sormani}.
\end{proof}


\subsection{Adapted Estimates}

Since at the post surgery times our space is an integral current space rather than a manifold, we can not apply the Theorem \ref{thm-sub-diffeo} right away. The integral current space we study, possesses nice properties that will allow us to apply a refined version of the Theorem \ref{thm-sub-diffeo}. This section is devoted to prove this refined version of the estimate on the SWIF distance. 

\begin{defn}\label{defn-related-metrics}
	Let $(M , g)$ be a Riemannian manifold (possibly disconnected and with boundary). Let $d : M \times M \to \R$ be a metric on $M$. We say that the metric $d$ is \textbf{related} to the Riemannian metric $g$ if for any smooth curve $C:(-r,r) \to M$, we have
\be
g \left( C'(0), C'(0) \right) = \left( \frac{d}{dt}\big|_{t=0} \;  d \left( C(t),C(0) \right) \right)^2.
\ee

\end{defn}

\begin{lem}\label{lem-related-metrics}
	Let $(M , g)$ be a Riemannian manifold (possibly disconnected and with boundary). Let $d : M \times M \to \R$ be a metric on $M$. If $d$ is related to $g$, then for any smooth curve $C:(-r,r) \to M$, we have
\be
	L_{g_i}(C)=L_{d_i}(C). 	
\ee
\end{lem}
\begin{proof}
	Consult any standard text on Metric Geometry for example ~\cite{BBI}.
\end{proof}

\begin{thm} \label{thm-sub-diffeo-new}
Given a pair of geodesic metric spaces $(Y_i, d_i)$, $i=1,2$ , containing integral current spaces $\left( X_i, d_i , T_i \right)$ , $i = 1,2$ with restricted metrics $d_i$, suppose there are  precompact subregions $U_i \subset \operatorname{set}(X_i)$ (possibly disconnected) that are Riemannian manifolds (possibly with boundary) with metrics $g_i$ such that the induced integral current spaces are
\be
	 \left( U_i, d_i , T_i \right) ,  \;\; i =1,2,
\ee
where, the metric $d_i$ on $U_i$ is restricted from $d_i$ on $X_i$ and,
\be
	T_i = \int_{U_i} \;\;, \;\; i=1,2.
\ee
and such that the metric $d_i$ is \textbf{related} (see Definition \ref{defn-related-metrics}) to the Riemannian metric $g_i$  for  $i=1,2.$

Assume there exist diffeomorphisms $\psi_i: U \to U_i$ such that
\be \label{thm-subdiffeo-new-1}
\psi_1^*g_1(V,V)
< (1+\epsilon)^2 \psi_2^*g_2(V,V) \qquad \forall \, V \in TU
\ee
and
\be \label{thm-subdiffeo-new-2}
\psi_2^*g_2(V,V) <
(1+\epsilon)^2 \psi_1^*g_1(V,V) \qquad \forall \, V \in TU.
\ee

We take the following \textbf{extrinsic} diameters,
\be \label{DU-new}
D_{U_i}= \sup \left\{ \diam_{X_i}(W): \, W\textrm{ is a connected component of } U_i\right\} \le \diam(X_i),
\ee
and define a hemispherical width,
\be \label{thm-subdiffeo-new-3}
a>\frac{\arccos(1+\epsilon)^{-1} }{\pi}\max\{D_{U_1}, D_{U_2}\}.
\ee

Let the distance distortion with respect to the outside integral current spaces be
\be \label{lambda-new}
\lambda=\sup_{x,y \in U}
|d_{X_1}(\psi_1(x),\psi_1(y))-d_{X_2}(\psi_2(x),\psi_2(y))|,
\ee
we define heights,
\be \label{thm-subdiffeo-new-4}
h =\sqrt{\lambda ( \max\{D_{U_1},D_{U_2}\} +\lambda/4 )\,}
\ee
and
\be \label{thm-subdiffeo-new-5}
\bar{h}= \max\{h,  \sqrt{\epsilon^2 + 2\epsilon} \; D_{U_1}, \sqrt{\epsilon^2 + 2\epsilon} \; D_{U_2} \}.
\ee

Then, the SWIF distance between the settled completions are bounded above as follows:
\begin{eqnarray*}
d_{\mathcal{F}}(X'_1, X'_2) &\le&
\left(2\bar{h} + a\right) \left(
\vol_m(U_{1})+\vol_m(U_2)+\vol_{m-1}(\partial U_{1})+\vol_{m-1}(\partial U_{2})\right)\\
&& + ||T_1||(X_1\setminus U_1) + ||T_2||(X_2\setminus U_2).
\end {eqnarray*}
\end{thm}

\begin{proof}
The theorem begins exactly as in the proof of thm-subdiffeo in ~\cite{Lakzian-Sormani} with a construction of an ambient space $Z$.

For every pair of corresponding diffeomorphic connected components $U_i^\beta$
of $U_i$, we can create a hemispherically defined filling bridge
$X'_\beta$ diffeomorphic to $U_i^{\beta_i} \times [0,a]$ with metric $g'_\beta$ satisfying (\ref{star-0})
by applying Proposition~\ref{prop-hem} and Proposition~\ref{prop-squeeze-Z} 
using the $a_i=a_i(\beta)$ defined there for the particular connected
component, $U_i^\beta$ and $D_i=D_{U_i}$.  Observe that all $a_i \le a$,
so $|t_1-t_2|=a$ will work for all the connected
components.   Any minimal geodesic $\gamma: [0,1]\to U_i^\beta$
of length $\le D_{U_i}\le \diam_{X_i}(U_i)$ satisfies (\ref{star-0}).

Let $X'$ be the disjoint unions of these bridges. $X'$ has a metric
$g'$ satisfying (\ref{star-1}). The boundary of $X'$ is $(U,g_1) \cup (U, g_2) \cup (\partial U \times [0 , a] , g')$.  Therefore,
\begin{eqnarray}\label{sdB}
\vol_m(X')&=&\sum_\beta \vol_m(X_\beta')\\
& \le& \sum_\beta  a(\vol_m(U_1^\beta)+\vol_m(U_2^\beta))\\
& \le& a (\vol_m(U_{1})+\vol_m(U_{2}))
\end{eqnarray}
and
\be\label{sdA}
\vol_{m}\left(\partial X'\setminus (\varphi_1(U_1)\cup
\varphi_2(U_2)\right) 
\le a \; (\vol_{m-1}(\partial U_1) + \vol_{m-1}(\partial U_2))
\ee
as in Proposition~\ref{prop-squeeze-Z}.

Since our regions are not necessarily convex, we cannot directly glue $X_i$ to $X'$ in order to obtain a distance preserving embedding.
We first need to glue isometric products
$U^\beta \times [0, \bar{h}]$ with cylinder metric $dt^2 +g_i$ to both ends of the filling bridges, to have all the bridges extended
by an equal length on either side.  This creates a Lipschitz manifold,
\be
X''=(U_1\times[0,\bar{h}])
\disjointunion_{U_1} X' \disjointunion_{U_2} (U_2 \times [0,\bar{h}]).
\ee
We then define $\varphi_i: U_i \to X''$ such that
\begin{eqnarray}
\varphi_1(x)&=&(x,0)\in U_1\times [0, \bar{h}]\\
\varphi_2(x)&=&(x,\bar{h})\in U_2\times [0, \bar{h}]
\end{eqnarray}

Then by (\ref{sdB}) and (\ref{sdA}), we have
\begin{eqnarray}\label{subdiffeo-B}
\vol_{m+1}(X'')&=& \vol_{m+1}(X') + \bar{h} (\vol_m(U_1) +\vol_m(U_2))\\
& \le& (a+ \bar{h}) (\vol_m(U_{1})+\vol_m(U_{2}))
\end{eqnarray}
and $\vol_{m}\left(\partial X''\setminus (\varphi_1(U_1)\cup
\varphi_2(U_2)\right)=$
\begin{eqnarray} \label{subdiffeo-A}
\qquad
&=&
\vol_{m}\left(\partial X'\setminus (\varphi_1(U_1)\cup
\varphi_2(U_2)\right)+
 \bar{h} (\vol_{m-1}(\partial U_1) +\vol_{m-1}(\partial U_2))\\
 &\le& (a+\bar{h}) (\vol_{m-1}(\partial U_1) + \vol_{m-1}(\partial U_2)).
\end{eqnarray}

Finally we glue $Y_1$ and $Y_2$
to the far ends of $X''$ along $\varphi_i(U_i)$ 
to create a connected length space. This is possible since $U_i$'s are manifolds and $Y_i$s are geodesic spaces.
\be
Z = \bar{X}_1 \disjointunion_{U_1} X'' \disjointunion_{U_2} \bar{X}_2
\ee
As usual, distances in $Z$ are defined by taking the infimum of
lengths of curves.   See Figure~\ref{fig-subdiffeo-new-1}.
Each connected component, $X''_\beta$ of $X''$ will be called the
filling bridge corresponding to $U^\beta$.

\begin{figure}[h] 
   \centering
   \includegraphics[width=5in]{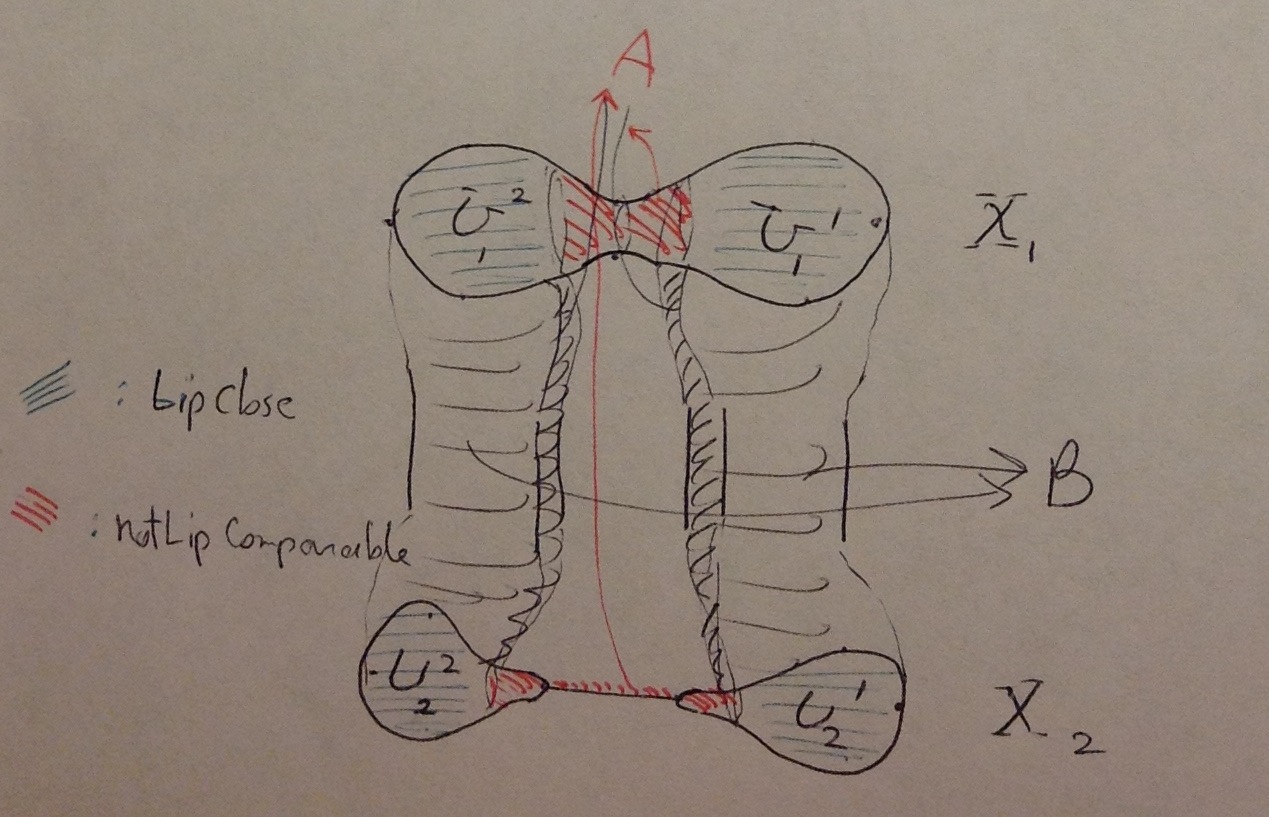}
   \caption{Creating $Z$ for Theorem~\ref{thm-sub-diffeo-new}.}
   \label{fig-subdiffeo-new-1}
\end{figure}

In the proof of Theorem 4.6 in ~\cite{Lakzian-Sormani}, it
is proven that $\varphi_1: Y_1 \to Z$ mapping $Y_1$ into
its copy in $Z$ is a distance preserving embedding. The proof there is given for manifolds but it can be easily adapted to our case since it only relies on the fact that our spaces are geodesic spaces and the fact that $g_i$ and $d_i$ are \textbf{related} (see Definition \ref{defn-related-metrics} ) on $U_i$s and both conditions are satisfied in our case.  The same argument shows that $\varphi_2:Y_2 \to Z$ is also a distance preserving embedding.


In order to bound the SWIF distance, we take $B^{m+1}=X''$ to be the filling current. Then the excess boundary is
\be
A^m=\varphi_1(X_1\setminus U_1) \cup\varphi_2(X_2\setminus U_2) \cup
\partial X'' \setminus (\varphi_1(U_1)\cup \varphi_2(U_2)).
\ee

Using appropriate orientations we have
\be\label{eq-pushforward}
	\varphi_{1\#}(T_1) - \varphi_{2\#}(T_2) = B^{m+1} + A^{m}.
\ee
Notice that (\ref{eq-pushforward}) is true since the $\operatorname{set} \left( \varphi_{i\#}(T_i) \right) = \varphi_i(X_i) $.

The volumes of the Lipschitz manifold parts have been computed in
(\ref{subdiffeo-A}) and (\ref{subdiffeo-B}).  So we get:
\begin{eqnarray*}
d_{\mathcal{F}}(X_1, X_2) &\le&
\vol_m(U_{1})\left( \bar{h} + a\right) +\vol_m(U_2)\left( \bar{h}+ a\right)\\
&& + \left( \bar{h} + a\right)\vol_{m-1}(\partial U_{1})+
\left( \bar{h}+ a\right)\vol_{m-1}(\partial U_{2})\\
&&+||T_1||(X_1\setminus U_1)+ ||T_2||(X_2\setminus U_2).
\end {eqnarray*}
\end{proof}



\section{Ricci Flow: Old and  New} \label{sect-RF}


\subsection{Hamilton's Ricci Flow}

Ricci flow is an evolution equation of the metric on a Riemannian manifold, introduced for the first time by Richard Hamilton in ~\cite{Hamilton-Ricci-Flow} given by the following weakly parabolic equation:
\be
	\frac{d}{dt} g(t) = -2 \Ric(t).
\ee

Hamilton proved that for the initial metric $g_0$ on a closed manifold $M$, the Ricci flow equation satisfies short time existence and uniqueness. \cite{Hamilton-Ricci-Flow} 

\begin{prop}\label{Lipschitz-distance-Ricci Flow}
Suppose $M$ is a closed manifold and let $g(t)$ be a solution to the Ricci flow equation on the time interval $[0,T]$. If
\be
	\| \Rm(t) \|	 \le K    \;\; \text{for all $t \in [0,T]$},
\ee
where $\| \Rm(t) \|$ is with respect to a fixed background metric $g_0$; then, for any $t_1 \le t_2$ and $V \in TM $ we have the following:
\be
	            e^{ - \; \sqrt{n} K (t_2 - t_1)}g(t_2) (V,V) \le   g(t_1) (V,V)  \le e^{\sqrt{n} K (t_2 - t_1)}g(t_2) (V,V),
\ee
and therefore,
\be
	e^{ - \sqrt{n} K (t_2 - t_1)}  \le \frac{d_{M , g(t_2)}(x , y)}{d_{M , g(t_1)}(x , y)} \le  e^{\sqrt{n} K (t_2 - t_1)}
\ee
\end{prop}
\begin{proof}
	c.f. \cite{MR2061425}.
\end{proof}

The above result also holds locally, namely:

\begin{prop}\label{Lipschitz-distance-Ricci Flow-local}
Suppose $M$ is a closed manifold and let $g(t)$ be a solution to the Ricci flow equation on the time interval $[0,T]$ and $\Omega \subset M$ an open subset of the manifold then, if
\be
	\sup_{x \in \Omega} \| \Rm(x,t) \|	 \le K    \;\; \text{for all $t \in [0,T]$},
\ee
where,  $ \| \Rm(x,t) \|$ is with respect to a fixed background metric $g_0$; then, for any $t_1 \le t_2$ and $V \in T\Omega $ we have the following:
\be
	            e^{ - \; \sqrt{n} K (t_2 - t_1)}g(t_2) (V,V) \le   g(t_1) (V,V)  \le e^{\sqrt{n} K (t_2 - t_1)}g(t_2) (V,V)
\ee
\end{prop}
\begin{proof}
	c.f. \cite{MR2061425}.
\end{proof}


\subsection{Ricci Flow Neckpinch}

In this section, we will review the results about neckpinch singularity obtained by Angenent-Knopf \cite{Neckpinch}\cite{AK-precise} and Angenent-Caputo-Knopf \cite{ACK}. We will repeat some of their Theorems and Lemmas from their work that we will be using later on in this paper. A nondegenerate neckpinch is a local type I singularity (except for the round sphere shrinking to a point) is arguably the best known and simplest example of a finite-time singularity that can develop through the Ricci flow. A nonegenerate neckpinch is a type I singularity whose blow up limit is a shrinking cylinder soliton. more precisely,

\begin{defn}\label{defn-neckpinch}
 a solution $\left( M^{n+1} , g(t) \right)$ of Ricci flow develops a neckpinch at a time $T< \infty$ if there exists a time-dependent family of proper open subsets $U(t) \subset M^{n+1}$ and diffeomorphisms $\phi(t) : \R \times \sphere^{n} \to U(t)$ such that $g(t)$ remains regular on $M^{n+1} \setminus U(t)$ and the pullback $\phi(t)^*\left( g(t) \right)$  on $\R \times \sphere^{n}$ approaches the \textbf{shrinking cylinder  soliton} metric 
\be
	ds^2 + 2(n-1)(T-t)g_{can}
\ee
\end{defn}

For the first time, Angenent-Knopf in \cite{Neckpinch} rigorously proved the existence of nondegenerate neckpinch on the sphere $\sphere^{n+1}$ in ~\cite{Neckpinch}. Their main result in \cite{Neckpinch} is as follows:

\begin{thm}

If $n  > 2$, there exists an open subset of the family of metrics on $\sphere^{n+1}$ possessing $\operatorname{SO}(n + 1)$ symmetries such that the Ricci fow starting at any metric in this set develops a neckpinch at some time $T < 1$. The singularity is rapidly-forming (Type I), and any sequence of parabolic dilations formed at the developing singularity converges to a shrinking cylinder soliton.
\be
	ds^2 + 2(n-1)(T-t)g_{can}.
\ee

This convergence takes place uniformly in any ball of radius
\be
	o \left( \sqrt{(T -t ) \log{\frac{1}{T - t }}} \; \right),
\ee
centered at the neck.

Furthermore, there exist constants $0 < \delta , C < \infty$ such that the radius $\psi$ of the sphere at distance   $\sigma$from the neckpinch is bounded from above by
\be
	\psi \le \sqrt{2(n-1)(T -t )} + \frac{C\sigma^2}{- \log{(T - t )} \sqrt{T -t }},
\ee
for $|\sigma| \le 2\sqrt{(T - t ) \log{(T - t )}} $, and, 
\be
	\psi \le C \frac{\sigma}{\sqrt{- \log{(T - t )}}}\sqrt{\log{\frac{\sigma}{-(T - t ) \log{(T -t )}}}},
\ee
for $  2\sqrt{(T - t ) \log{(T - t )}}   \le   \sigma \le (T - t)^{\tfrac{1}{2} - \delta }$
\end{thm}

The class of initial metrics for which we establish "neckpinching" is essentially described by three conditions:
(i) the initial metric should have positive scalar curvature, (ii) the sectional curvature of the initial
metric should be positive on planes tangential to the spheres $\{x\} \times  \sphere^n$, and (iii) the initial metric should
be " sufficiently pinched".

\begin{lem}\label{bounding-curvature-off-neckpinch}
There is a constant $C$ depending on the solution $g(t)$ such that:
\be
	\| \Rm \| \le \frac{C}{\psi^2}
\ee

\end{lem}

\begin{proof}
	See the Lemma 7.1 in ~\cite{Neckpinch}.
\end{proof}

\subsection{Diameter Bound}
The following diameter bound argument is necessary before we can talk about the intrinsic flat convergence. Proposition \ref{prop-diameter-bound} in below is taken from ~\cite{AK-precise}. We are also including the proof of this Proposition from ~\cite{AK-precise} for completeness of exposition because we need the estimates in the proof as well as the result itself.

\begin{prop}[~\cite{AK-precise}]\label{prop-diameter-bound}
Let $(\sphere^{n+1} , g(t))$ be any $\operatorname{SO}(n+1)$ invariant solution of the Ricci Flow such that $g(0)$ has positive scalar curvature and positive sectional curvature on planes tangential to the spheres $x \times \sphere^n$, assume that in the language of ~\cite{Neckpinch}, each $g(t)$ has at least two bumps for all $t < T$. Let $x = a(t)$ and $y = b(t)$ be the locations of the left- and right- most bumps, and assume that for all $t < T$, one has $\psi (a(t) , t) \ge c$ and $\psi (b(t) , t) \ge c$ for some constant $c>0$. If $g(t)$ becomes singular at $T < \infty$, then $\diam (\sphere^{n+1} , g(t))$ remains bounded as $t \nearrow T$.  
\end{prop}

\begin{proof}(~\cite{AK-precise})
	By Proposition 5.4 of ~\cite{Neckpinch}, the limit profile $\psi (\; . \;  , T)$ exists. let $a(t) \to a(T)$ and $b(t) \to b(T)$. By lemma 5.6 of ~\cite{Neckpinch}, the Ricci curvature is positive (and so the distances are decreasing)  on $(-1 , a(t)]$ and $[b(t) , 1)$. Hence it will suffice to bound $d_{(M , g(t))} (x_1 , x_2)$ for arbitrary $x_1 < x_2$ in $(a(T) - \epsilon , b(T) + \epsilon) \subset (-1 , 1)$.
	
Equations  (5) and (11) of ~\cite{Neckpinch} imply that 
\begin{eqnarray}
	\frac{d}{dt} d_{M,g(t)}(x_1 , x_2) &=& \frac{d}{dt} \int_{x_1}^{x_2} \; \phi (x,t) \; \mathrm{d}x  \notag \\ &=& n \int_{s(x_1)}^{s(x_2)} \; \frac{\psi_{ss}}{\psi} \; \mathrm{d}s \\&=& n \left\{  \frac{\psi_s}{\psi} \biggr|^{s(x_2)}_{s(x_1)}  + \int_{s(x_1)}^{s(x_2)} \; \left( \frac{\psi_s}{\psi} \right)^2 \; \mathrm{d}s \right\}. \notag
\end{eqnarray}

Proposition 5.1 of ~\cite{Neckpinch}, bounds $\psi_s$ uniformly, while lemma 5.5 shows that the number of bumps and necks are non-increasing in time. It follows that:
\begin{eqnarray}
	\int_{s(x_1)}^{s(x_2)} \; \left( \frac{\psi_s}{\psi} \right)^2 \; \mathrm{d}s &\le& C \int_{s(x_1)}^{s(x_2)} \;  \frac{|\psi_s|}{\psi^2}  \; \mathrm{d}s \notag \\ &\le& C \left[ \frac{1}{\psi_{min}(t)} - \frac{1}{\psi_{max}(t)}  \right] \\ &\le& \frac{C}{\psi_{min}(t)}. \notag
\end{eqnarray}

 Hence lemma 6.1 of ~\cite{Neckpinch} lets us conclude that:
 \be
 	\left| \frac{d}{dt} d_{M,g(t)}(x_1 , x_2) \right| \le \frac{C}{\sqrt{T -t}},
 \ee
which is obviously integrable. 
\end{proof}

\begin{lem}\label{Equatorial-only-pinching}
	If the diameter of the solution $g(t)$ stays bounded as $t \nearrow T$ then $\psi (s,T) > 0$ for all $ 0 < s < D/2$, where $D$ is defined as $D = \lim_{t \nearrow T} \psi (x_*(t) , t )$ in which $x_*(t)$ denotes the location of the right bump.  	
\end{lem}

\begin{proof}
	See the Lemma 10.1 of ~\cite{Neckpinch}.
\end{proof}


\subsection{Smooth Forward Evolution of Ricci Flow}\label{sect-smooth-forward-evolution}

In this section we will review the results obtained by Angenent, Caputo and Knopf in ~\cite{ACK} about the neckpinch on the sphere in any dimension and their attempt to find a canonical way to perform surgery at the singular time (in this case, finding a limit for surgeries whose scale of the surgery is going to zero.) 

Consider the degenerate metric $g(T)$ resulted from the Ricci Flow neckpinch on $\sphere^{n+1}$ as described earlier. Angenent-Caputo-Knopf in ~\cite{ACK} construct the smooth forward evolution of Ricci flow by regularizing the pinched metric in a small neighborhood of the pinched singularity (of scale $\omega$) and hence producing a smooth metric $g_\omega$. Notice that performing surgery at a small scale $\omega$ produces two disjoint Ricci Flows. For simplicity, we only consider one of these resulting spheres and then we assume that the north pole is the future of the neckpinch point singularity. By the short time existence of Ricci flow, for any small scale $\omega$, the flow exists for a short time depending on $\omega$. Using the asymptotics for Ricci flow neckpinch derived in Angenent-Knopf ~\cite{Neckpinch}, they find a lower bound for the maximal existence time $T_\omega$ of the Ricci flows $g_\omega (t)$ with initial metrics $g_\omega$. Of course, as $\omega \to 0$, one has $g_\omega \to g(T)$ away from the point singularity. Now the question is if also the resulting Ricci flow solutions $g_\omega (t)$ admit a limit as $\omega \to 0$. They prove that this is in fact the case by proving bounds on the curvature off the singularity and then proving a compactness theorem. This limit flow is called the smooth forward evolution of Ricci flow out of a neckpinch singularity. 

Angenent-Caputo-Knopf ~\cite{ACK} show that a smooth forward evolution of Ricci flow out of a neckpinch singularity comes from (via a change of variable) a positive solution of the following quasilinear PDE: 
 \be
	v_t = vv_{rr} - \frac{1}{2}v^2  + \frac{n - 1 - v}{r}v_r + \frac{2(n - 1)}{r^2}(v - v^2) 
\ee
with the singular initial data:
\be\label{eq-initial-data}
	v_{init}(r) = [1 + o(1)]v_0(r) \;\; \text{as} \;\; r \searrow 0,
\ee
where 
\be
	v_0(r) \doteqdot \frac{\tfrac{1}{4}(n-1)}{ - \log{r}}. 
\ee

One notices that away from the point singularity, any smooth forward evolution of (\ref{eq-initial-data}) has to satisfy
\be
\lim_{t \searrow 0} v(r,t) = v_{init}(r)
\ee
 
They prove that the only way, a solution to this equation can be complete is if $v$ satisfies the smooth boundary condition $v(0,t) = 1$ which is incompatible with the fact that $\lim_{r \searrow 0} v_{init}(r) = 0$.  Roughly speaking, this means that for any forward evolution of Ricci flow,  $v$ immediately jumps at the singular hypersurface $\{0\} \times \sphere^n$, yielding a compact forward evolution that replaces the singularity with a smooth $n$-ball by performing a surgery at scale $0$. See Figure \ref{fig-ACK}.

For given small $ \omega> 0 $, they split the manifold $\sphere^{n+1}$ into two disjoint parts, one of which is the small neighborhood $\mathcal{N}_\omega$ of the north pole in which  $\psi_T(s) < \rho_{*} \sqrt{\omega}$ (See ~\cite{ACK} for details about this construction.) 

They keep the metric unchanged on $\sphere^{n+1} \setminus \mathcal{N}_\omega $. Within $\mathcal{N}_\omega$, they take $g_\omega$ to be a metric of the form 
\be
g_\omega = (ds)^2 +  \psi_\omega(s)^2g_{can},
\ee
where $\psi_T$ is a monotone when $\psi_\omega(s) \le \rho_{*} \sqrt{\omega}$ . Monotonicity of $\psi$ in $\mathcal{N}_\omega$, allows them to perform the change of variables $r =  \psi(s)$. In this new coordinate, one sees that $g_\omega$ is of the form:
\be
	g_\omega = \frac{dr^2}{v_\omega(r)} + r^2g_{can}.
\ee

Angenent-Caputo-Knopf then proceed to apply a maximum principle to find sub- and supersolutions of this equation which bound all the positive solutions. A detailed analysis of these bounds enables them to prove curvature estimates that are required in their compactness theorem (See ~\cite{ACK} for further details.)

The main theorem of Angenent-Caputo-Knopf's work about the smooth forward evolution of the Ricci flow past the singularity time is as follows:

\begin{thm}[~\cite{ACK}]\label{ACK-main-thm}
For $n>2$, let $g_0$ denote a singular Riemannian metric on $\sphere^{n+1}$ arising as the limit as $t \nearrow T$ of a rotationally symmetric neckpinch forming at time $T $. Then there exists a complete smooth forward evolution 
\be
\left( S^{n+1} , g(t) \right)  \;\;  \text{for}\;\;  T < t < T_1,
\ee
of $g(T)$ by Ricci Flow. Any complete smooth forward evolution is compact and satisfies a unique asymptotic profile as it emerges from the singularity. In a local coordinate $ 0 < r < r_* \ll 1$ such that the singularity occurs at $r = 0$ and the metric is
\be
	g(r,t) = \frac{dr^2}{v(r,t)} + r^2 g_{can}
\ee

This asymptotic profile is as follows:\\

{\bf Outer Region:}  For $ c_1\sqrt{t - T} < r < c_2$, one has:
\be
	v(r,t) = [1 + o(1)] \frac{n-1}{- 4 \log{r}} \left[  1 + 2(n-1)\tfrac{t-T}{r^2}  \right] \;\; \text{uniformly as }\;\; t \searrow T.
\ee

{\bf Parabolic Region:} Let $ \rho =  \frac{r}{\sqrt{t - T}} $ and $\tau = \log{(t - T)}$; then for $ \frac{c_3}{\sqrt{-\tau}} < \rho < c_4$, one has:
\be
	v(r,t) = [1 + o(1)] \frac{n-1}{- 2 \tau} \left[  1 + \frac{2(n-1)}{\rho^2}  \right] \;\; \text{uniformly as }\;\; t \searrow T.
\ee

{\bf Inner Region:} Let $\sigma = \sqrt{- \tau}\rho = \sqrt{\frac{- \tau}{t-T}}r $; then for $0 < \sigma < c_5$, one has:
\be
	v(r,t) = [1 + o(1)] \mathcal{B}\left( \frac{\sigma}{n-1} \right)   \;\; \text{uniformly as }\;\; t \searrow T.
\ee 
where $\frac{d\sigma^2}{\mathcal{B}(\sigma)} + \sigma^2g_{can}$ is the Bryant soliton metric.

\end{thm}

\begin{figure}[h] 
   \centering
   \includegraphics[width=5in]{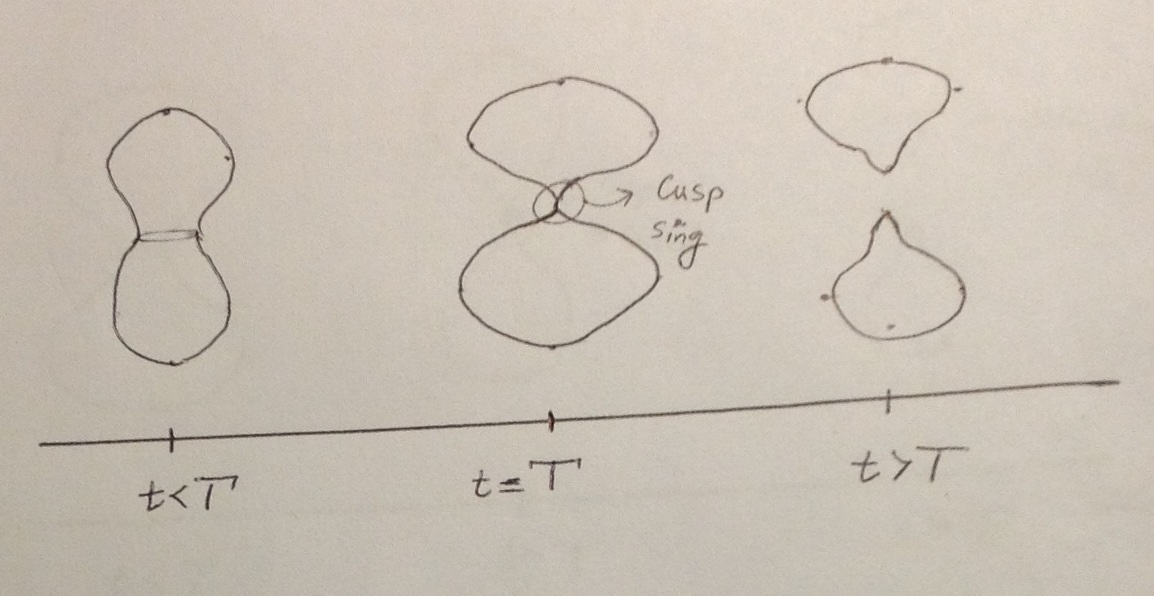}
   \caption{Angenent-Caputo-Knopf Ricci Flow Through Neckpinch Singularity}
   \label{fig-ACK}
\end{figure}

\begin{rmrk}\label{rmrk-sing-metric}
In the work of Angenent-Caputo-Knopf ~\cite{ACK}, assumptions on the singular initial metric $ g(T) = ds^2 + \psi_T(s)^2 g_{can} $ is as follows:

\begin{eqnarray}
	&&(M1) \;\;   \psi_T(s) > 0 \;\; \text{for all}\;\; s \in J \notag \\ \notag
	&&(M2) \;\;    \psi_T(0) = \psi_T(l) = 0 \\ \notag
	&&(M3)  \;\;   \psi_T'(l) = -1     \\ \notag
	&&(M4) \;\;    \psi_T(s)^2 = \left( \frac{n-1}{4} + o(1)  \right) \frac{s^2}{- \log{s}}  \;\; (s \searrow 0) \\
	&&(M5) \;\;    \psi_T(s)\psi_T'(s) = \left( \frac{n-1}{4} + o(1)  \right) \frac{s}{- \log{s}}  \;\; (s \searrow 0) \\ \notag
	&&(M6)  \;\;   | \psi_T'(s) | \le 1 \;\; (0 < s < l) \\ \notag
	&&(M7)  \;\;   \exists r_\# > 0 , \psi_T'(s) \neq 0 \;\; \text{whenever} \;\; \psi_T(s) < 2 r_\# \\ \notag
	&&(M8)  \;\;   \exists \mathcal{A} \; \forall s \in J , |a_T(s)| \le \mathcal{A}\;\; ( \text{where} \; a_0(s) = \psi_T'\psi_T'' - \psi_T'^2 + 1 ) 
\end{eqnarray}

See ~\cite{Neckpinch} and \cite{ACK} for details. 
\end{rmrk}


\section{Continuity of Angenent-Caputo-Knopf's Smooth Forward Evolution}\label{sect-main}

In this section, we will study the continuity of Ricci Flow through singularities under the intrinsic flat distance. Though there is a caveat to this claim; At the post surgery times, our flow consists of two separate Ricci Flows on two disjoint manifolds that are not canonically embedded into a space and hence there is no apriori metric that makes this disjoint union of manifolds into a metric space. So in order to makes sense of different notions of convergence at post surgery times, we need to first define a metric on this disjoint union. Here we will assume that the two parts at the post surgery time $t$ are connected by a thread of length $L(t)$ which joins the future (two points) of the singular point. We also notice that since the thread is one dimensional, it does not contribute in the settled completion of the resulting current space. Also naturally we require 
\be
	\lim_{t \searrow T} L(t) = 0
\ee

Our main theorem is the following:

\begin{thm}
	The compact Ricci flow through singularities obtained by smooth forward evolution of the Ricci flow out of a neckpinch singularity on the $n+1$-sphere (as in ~\cite{ACK}) is continuous under the Sormani - Wenger Intrinsic Flat distance (SWIF).
\end{thm}
\begin{proof}
 Combining Lemmas ~\ref{lem-pre-surgery} and ~\ref{lem-post-surgery} completes the proof.
\end{proof}


\subsection{Smooth Ricci Flow}

Here we estimate the Intrinsic Flat distance between two times of a compact smooth Ricci flow defined on $[0,T)$ which will be easily derived from Theorem \ref{thm-sub-diffeo}. Since as $t \to t_0  \in [0,T)$ we have
\be
	g(t) \to g(t_0)
\ee
uniformly in smooth norm, it is not surprising that we must also have:
\be
	d_{\mathcal{F}}\Big( \left(M , g(t) \right), \left(M,g(t_0) \right) \Big) \to 0.
\ee
as $t \to t_0$. In fact, we have

\begin{lem}\label{lem-smooth}
	Suppose $\left(M^n , g(t) \right)$ is a smooth solution of Ricci flow on a closed manifold $M^n$ defined on the time interval $[0,T)$. Then,  for any $t_1 , t_2 \in [0,T)$ 
\be
	d_{\mathcal{F}}\Big( \left(M , g(t_1) \right), \left(M,g(t_2) \right) \Big) \le 
	  \frac{\arccos  \sqrt{e^{\sqrt{n} C (t_1 - t_2)}}}{\pi} \max \left\{ \diam(M , g(t_1)) , \diam(M , g(t_2)) \right\}
\ee
where $C$ is a uniform upper bound for $||\Rm||$.

\end{lem}

\begin{proof}
Since the flow is smooth on $[0,T)$, for any compact sub-interval $J \subset [0,T)$ we have
\be
	\sup_{M^n \times J }||\Rm||\le C=C(J).
\ee
with respect to the initial metric $g_0$ on $M^n$. 

By applying Theorem \ref{Lipschitz-distance-Ricci Flow}, we get
\be \label{eq-smooth-1}
	            e^{ - \; \sqrt{n} C (t_2 - t_1)}g(t_2) (V,V) \le   g(t_1) (V,V)  \le e^{\sqrt{n} C (t_2 - t_1)}g(t_2) (V,V),
\ee

Let $\epsilon=  \sqrt{e^{\sqrt{n} C (t_2 - t_1)}} - 1$ then (\ref{eq-smooth-1}) gives
\be
	g(t_1) (V,V)  \le \left(1 + \epsilon \right)^2 g(t_2) (V,V),
\ee
and
\be
	g(t_2) (V,V)  \le \left(1 + \epsilon \right)^2 g(t_1) (V,V).
\ee

Finally using Proposition \ref{prop-squeeze-Z}, we get:
\be
	d_{\mathcal{F}}\Big( \left(M , g(t_1) \right), \left(M,g(t_2) \right) \Big) \le 
	  \frac{\arccos  \sqrt{e^{\sqrt{n} C (t_1 - t_2)}}}{\pi} \max \left\{ \diam(M , g(t_1)) , \diam(M , g(t_2)) \right\}
\ee

\end{proof}


\subsection{Ricci Flow Through the Singularity as an Integral Current Space}

Let $\left( M , g(t) \right)$ be the Angenent-Knopf's example. At any time $t < T$, $\left( M , g(t) \right)$ is a Riemannian manifold. As before, taking the current structure
\be
	T = \int_{M}
\ee
on $M$, one can think of $\left( M , g(t) \right)$ as an integral current space. 

It is well-known that any point $p$ in any Riemmanian manifold $M$ is asymptotically Euclidean hence 
\be
	\Theta_*(p) = 1,
\ee 
and as a result, $\operatorname{set}(M) = M$.

At the singular time $t = T$, the metric $g(T)$ is degenerate at the level set $\{0\}\times \sphere^n$ but nonetheless still gives rise to the distance metric $d$ on the pinched sphere by minimizing the length $L(\gamma) = \int_{\gamma} \left( g(\gamma'(s) , \gamma'(s) \right)^{\frac{1}{2}} \mathrm{d}s$ along curves as usual. The pinched sphere $\left(M , g(T) \right)$ is again an integral current space. One way to see this is that one observes that the pinched sphere is a union of two $C^1$ - manifolds and the singular point which is of course of measure $0$. See Figure \ref{fig-ACK-IC}.

\begin{figure}[h] 
   \centering
   \includegraphics[width=5in]{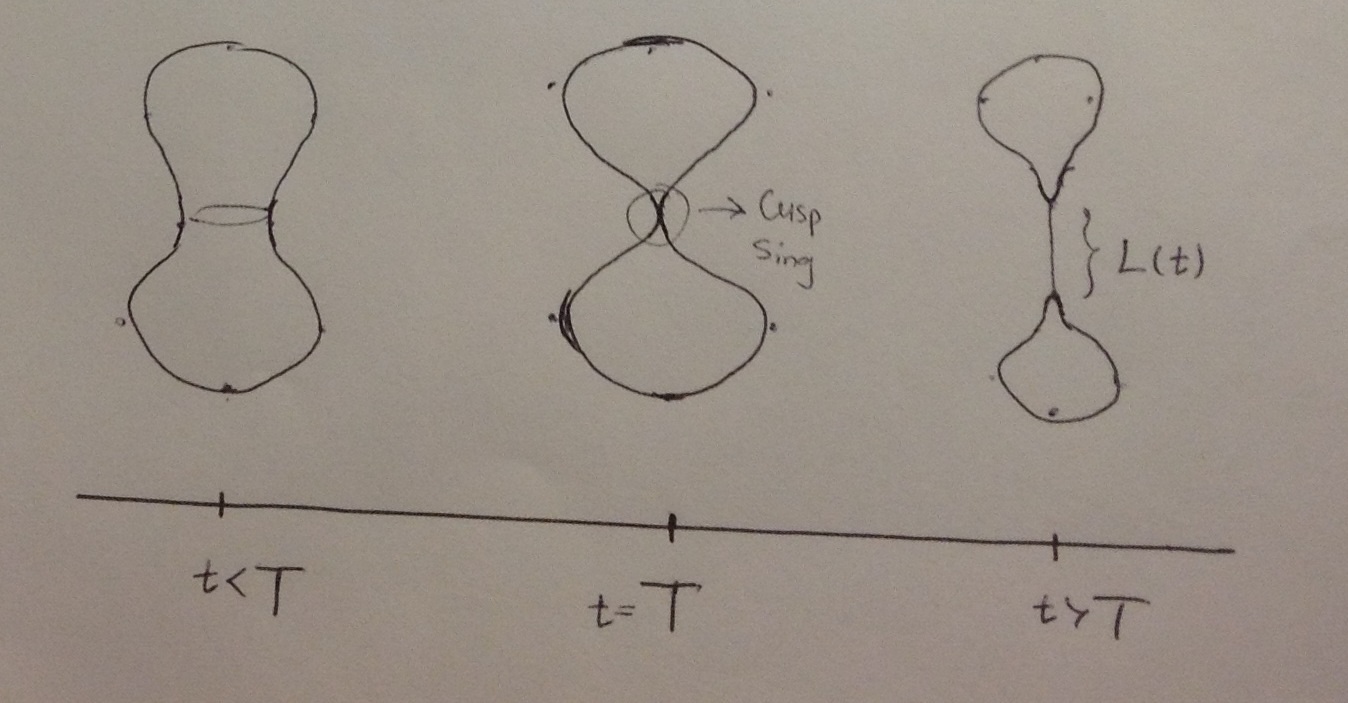}
   \caption{ The Ambient Space for Angenent-Caputo-Knopf's Ricci Flow Through Neckpinch Singularity (Not \emph{Settled}).}
   \label{fig-ACK-IC}
\end{figure}

The caveat here is that when considering the singular $\left(M , g(T) \right)$ as an integral current space, by definition, we need to only consider the settled completion   i.e. the points with positive density (See Figure \ref{fig-ACK-Set-IC}). The Lemma \ref{lem-density-sing} below computes the settled completion.

\begin{lem}\label{lem-density-sing}
	Let $p$ be the singular point in the pinched sphere $\left( M,g(T) \right)$ then,
\be	
	\operatorname{set}(M,g(T)) = M \setminus \{p\}.
\ee 
\end{lem}

\begin{proof}
According to ~\cite[Table 1]{ACK} ( or Remark ~\ref{rmrk-sing-metric} in this paper), at the singular time $T$, for the singular metric
\be
	g(T) = ds^2 + \psi_T(s)g_{can}
\ee
we have:
\be
	\psi_T(s) \sim s |\ln s|^{-\frac{1}{2}} \;\; as \;\; s \to 0
\ee
therefore, one computes
\begin{eqnarray}
	\Theta_*(p) &=& \liminf_{r \to 0} \frac{\vol (B(p,r))}{r^{n+1}} \\ &\le& C \liminf_{r \to 0} \frac{\int_0^r s \left(s |\ln s|^{-\frac{1}{2}} \right)^n }{r^{n+1}} \\ &\le& C \liminf_{r \to 0} \frac{ r^{n+1} \left(|\ln r|^{-\frac{1}{2}} \right)^n }{r^{n+1}}  \\ &=& C \liminf_{r \to 0} \left(|\ln r|^{-\frac{1}{2}} \right)^n  \\&=&0,
\end{eqnarray}
which means that the singular point $p \not\in \operatorname{set}(M,g(T))$.

For all the regular points $x \in M\setminus \{p\}$, we again have
\be
	\Theta_*(x) = 1.
\ee
This concludes the proof.
\end{proof}

\begin{figure}[h] 
   \centering
   \includegraphics[width=5in]{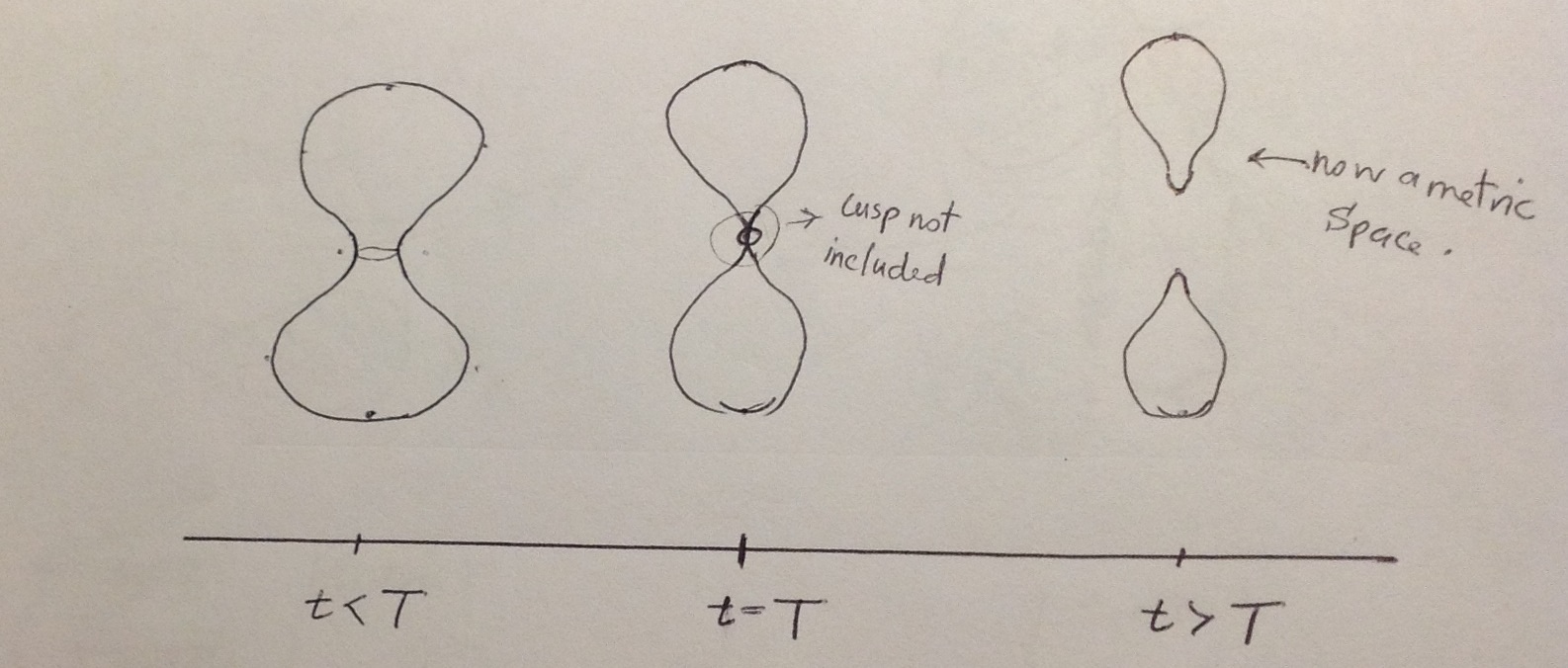}
   \caption{Angenent-Caputo-Knop's Ricci Flow Through Neckpinch Singularity Viewed as a \emph{Settled} Flow of Integral Current Spaces.}
   \label{fig-ACK-Set-IC}
\end{figure}

At the post surgery times $t > T$, the Flow is a result of the smooth forward evolution as in ~\cite{ACK} and hence consists of two separate smooth pointed Ricci Flows $\left(M_i , g_i(T) , p_i \right)$ $i=1,2$ obtained by regularizing the metric at the singular time. Points $p_i$ are just the future of the singular point $p$. Again in order to work in the framework of integral current spaces, we first need to make the post surgery flow into a metric space and also define an appropriate current structure on it. One way to make a current space out of the disjoint union $M_1 \disjointunion M_2$ as suggested by Knopf is to attach them by a thread. Another way is to define the metric using techniques from optimal transport which is described in the work of Author with Munn~\cite{Lakzian-Munn}. Here, we will focus on the thread approach. 

We should clarify that the added thread joins $p_1$ to $p_2$, has length $L(t)$ which is continuous with respect to $t$ and satisfies 
\be
	\lim_{t \searrow T} L(t) = 0.
\ee


\subsection{Continuity and Volume Convergnce as $t \nearrow T$}

Consider the neckpinch on the sphere with bounded diameter. Our goal is to use Theorem \ref{thm-sub-diffeo-new} to find an estimate on the intrinsic flat distance between the sphere prior to the singular time and the pinched sphere.

\begin{lem}\label{lem-metric-distortion}
	Let $\left(M , g(t) \right)$ be the Ricci flow on the $n+1$-sphere with a neckpinch singularity at time $T$. Then we have the following metric distortion estimate:
\be
	\left|  d_{(M,g(T))}(x_1 , x_2) -  d_{(M,g(t))}(x_1 , x_2) \right| \le C \sqrt{T - t}.	
\ee
\end{lem}

\begin{proof}
From the proof of Proposition \ref{prop-diameter-bound}, we have:
 \be
 	\left| \frac{d}{dt} d_{(M,g(t))}(x_1 , x_2) \right| \le \frac{C}{\sqrt{T -t}}.
 \ee

Hence a simple integration shows that
\be
	\left|  d_{(M,g(T))}(x_1 , x_2) -  d_{(M,g(t))}(x_1 , x_2) \right| \le C \sqrt{T - t}.
\ee

\end{proof}

\begin{lem}\label{lem-pre-surgery}
	For the neckpinch on the $n+1$-sphere, we have:
\be
	\lim_{t \nearrow T} d_{\mathcal{F}} \Big( (M,g(t)) , (M , g(T)) \Big) = 0.
\ee

\end{lem}

\begin{proof}
Notice that From lemma \ref{Equatorial-only-pinching}, we know that the pinching occurs only at the equator given by $x = 0$.  Let $M = \sphere^{n+1}$ and $S$ be the $n+1$ sphere, the singular set $\{x =0\}$ respectively. Let $U_{j}$ be the exhaustion of $M \setminus S$ defined by:
\be
	U_j = \left\{  (x,\theta) \in M :  |x| \ge 1/j  \right\}.
\ee

Each $U_j$ consists of two connected components $U_j^\beta$ , $\beta = 1,2$. Fix $j$, then from Lemma \ref{bounding-curvature-off-neckpinch}, There is constant $C_j$ depending on $j$ and the solution $g(t)$ such that :
\be
\sup_{U_j} \| \Rm \| \le C_j.
\ee

Therefore by Proposition \ref{Lipschitz-distance-Ricci Flow-local}, we have:
\be
	e^{ - \; \sqrt{n} C_j (t_2 - t_1)}g(t_2) (V,V) \le   g(t_1) (V,V)  \le e^{\sqrt{n} C_j (t_2 - t_1)}g(t_2) (V,V),
\ee
and letting $t_2 \nearrow T$, we get:
\be
	e^{ - \; \sqrt{n} C_j (T - t)}g(T) (V,V) \le   g(t) (V,V)  \le e^{\sqrt{n} C_j (T - t)}g(T) (V,V),
\ee
where $V \in TU_j $.

Therefore, in the setting of the Theorem \ref{thm-sub-diffeo-new}, we let $\epsilon_{tj} =  e^{\sqrt{n} C_j (T - t)} - 1$. We also need to compute the distortion $\lambda_{t,j}$ between these two length spaces which is defined as:
\be \label{lambda}
\lambda_{tj} = \sup_{x,y \in U_j} \left| d_{(M , g(t))}(\psi_1(x),\psi_1(y))-d_{(M , d_T)}(\psi_2(x),\psi_2(y)) \right|.
\ee

Let $x \in U_j^1$ and $y \in U_j^2$ then from Proposition \ref{prop-diameter-bound}, we obtain the following estimates on the distortion of distances which is independent of $j$; i.e.
\be\label{lambda-estimate}
	\lambda_{tj} \le C \sqrt{T - t }.
\ee

As in Theorem \ref{thm-sub-diffeo-new}, let
\be \label{thm-sub-diffeo-4}
h_{tj} =\sqrt{\lambda_{tj} ( \max\{D_{U_j^1},D_{U_j^2}\} +\lambda_{tj}/4 )\,},
\ee
and,
\be \label{thm-sub-diffeo-new-5}
\bar{h}_{tj}= \max\{h_{tj},  \sqrt{\epsilon_{tj}^2 + 2\epsilon_{tj}} \; D_{U_j^1}, \sqrt{\epsilon_{tj}^2 + 2\epsilon_{tj}} \; D_{U_j^2} \}.
\ee

For fixed $j$, as $t \nearrow T$, we have:
\be
	\epsilon_{tj} \to 0    \;\; \text{and} \;\; \lambda_{tj} \to 0.
\ee	

Therefore for all $j$:
\be
	\lim_{t \nearrow T} d_{\mathcal{F}}\left( (M, g(t)), (M, d_T) \right) \le   \vol_t (M \setminus U_j)
+ \vol_T(M \setminus U_j),
\ee

Also since the diameter stays bounded as $t \nearrow T$, one sees that as $j \to \infty$, 
\be\label{eq-vol-complement}
	\vol_t (M \setminus U_j) \;\;  \text{and} \;\; \vol_T(M \setminus U_j) \to 0.
\ee

Therefore,
\be
	\lim_{t \nearrow T} d_{\mathcal{F}}\left( (M, g(t)), (M, g(T)) \right) = 0.
\ee
\end{proof}

\begin{rmrk}\label{remark-volume-convergence-1}
Since the diameter stays bounded as $t \nearrow T$, (\ref{eq-vol-complement}) implies the volume convergence
\be
	\vol \left( M , g(t) \right) \to \vol \left( M , g(T) \right)
\ee
\end{rmrk}


\subsection{Continuity and Volume Convergence as $t \searrow T$}

To complete the proof of the continuity of the Smooth Forward Evolution of the Ricci Flow out of neckpinch singularity, we need to also prove the continuity as the time approaches the singular time from the post surgery times. For simplicity we let $T=0$ then, post surgery times will correspond to positive values of $t$.

Our flow at the positive time $t>0$ consists of two pointed smooth Ricci flows $\left( M_1 , g_1(t) , p_1 \right) $ and $\left( M_2 , g_2(t), p_2 \right)$ both modeled on the $n+1$-sphere and a thread of length $L(t)$ joining $p_1$ to $p_2$. As before, we let $\left( M , g(t) \right)$ denote the pre-surgery Ricci flow and $\left( M , g(T) \right)$ to be the singular space at the singular time $T$. Also we let $X = \left( M_1  \cup M_2 , D(t) , T(t) \right)$ be the current space associated to the post surgery time $t$, where

$$D(t)(x,y) = \begin{cases}
d_1(t)(x,y) & x,y  \in M_1 \\
d_2(t)(x,y) & x,y  \in M_2\\
L(t) + d_1(t)(x,p_1) + d_2(t)(y,p_2) & x\in M_1 \; \;and\;\; y \in M_2,
\end{cases}$$
where, $d_i(t)$ is the metric induced by the Riemannian metric $g_i(t)$ on $M_i$.

\begin{lem}\label{lem-post-surgery}
	If $\left( M_i , g_i(t) , p_i \right)$ $i=1,2$ represent the two parts of the post-surgery Ricci flow ($t > T = 0$) obtained by smooth forward evolution out of a neckpinch singularity and if $L(t)$ is the length of the thread joining $p_1$ and $p_2$ at time $t$ with
	\be
		\lim_{t \searrow 0} L(t) = 0,
	\ee
then, letting $X = M_1 \cup M_2$, we have:
\be
	\lim_{t \searrow 0} d_{\mathcal{F}} \Big( \left(X , D(t), T(t)\right) , \left( M, g(T) \right) \Big) = 0.
\ee

\end{lem}

\begin{proof}
Similar to the proof of the continuity for pre-surgery times, we need to find proper diffeomorphic open subsets. For fixed small $\omega >0$ consider the open subsets $ \mathcal{N}^i_\omega \subset M_i $ for  as defined in Section ~\ref{sect-smooth-forward-evolution}. Let $U_1$ be the open subset of $M$ defined by: 
\be
	U_1 = M \setminus \left( \bar{ \mathcal{N}}^1_\omega \cup \setminus  \bar{ \mathcal{N}}^2_\omega \right) ,
\ee
therefore,$U_1$ is comprised of two connected components $U_1^\beta$ for $\beta = 1,2$.

And let $U_2$ be the open subset of $M_1 \cup L(t) \cup M_2$ defined as 
\be
	U_2 = \left(M_1 \setminus \bar{ \mathcal{N}}^1_\omega \right) \cup \left(M_2 \setminus  \bar{ \mathcal{N}}^2_\omega \right).
\ee

Then obviously, these two open sets are diffeomorphic through diffeomorphisms between their correspondng connected components:
\begin{eqnarray}
	&& \psi_1 : W_1 \to M_1 \setminus \bar{ \mathcal{N}}^1_\omega \\
	&&\psi_2 : W_2 \to M_2 \setminus \bar{ \mathcal{N}}^2_\omega .
\end{eqnarray}

Now let $\epsilon(t)$ be the smallest positive number for which
\be
\psi_i^*g_i(t)(V,V)
< (1+\epsilon(t))^2 g(T)(V,V) \qquad \forall \, V \in TW_i
\ee
and
\be
\psi_i^*g_i(t)(V,V) <
(1+\epsilon(t))^2 g(T)(V,V) \qquad \forall \, V \in TW_i.
\ee
then, by the construction of the Smooth Forward Evolution as seen in Section ~\ref{sect-smooth-forward-evolution}, as $t \searrow 0$ , the metrics $\psi_i^*g_i(t)$ smoothly converge to $g(T)$ on $W_i$ therefore,
\be
	\lim_{t \searrow 0}\epsilon(t) = 0.
\ee

Now let $\omega_j > 0$ be a sequence for which
\be
	\lim_{j \to \infty}\omega_j = 0
\ee
and consider the length distortions:
\be \label{lambda-post}
\lambda_{tj} = \sup_{x,y \in U_1} \left| d_{(X , D(t))}(\psi_1(x),\psi_1(y)) - d_{(M , d_T)}(x, y) \right|.
\ee

Then, 
\be
	\lambda_{tj} \le L(t) + \Big((1+\epsilon(t))^2 - 1 \Big) \Big(\diam\left(M_1 , g_1(t) \right) + \diam \left( M_2 , g_2(t)\right)  \Big)
\ee

As in Theorem \ref{thm-sub-diffeo}, we define
\be \label{thm-sub-diffeo-4}
h_{tj} =\sqrt{\lambda_{tj} ( \max\{D_{U_j^1},D_{U_j^2}\} +\lambda_{tj}/4 )\,},
\ee
and,
\be \label{thm-sub-diffeo-5}
\bar{h}_{tj}= \max\{h_{tj},  \sqrt{\epsilon_{tj}^2 + 2\epsilon_{tj}} \; D_{U_j^1}, \sqrt{\epsilon_{tj}^2 + 2\epsilon_{tj}} \; D_{U_j^2} \}.
\ee

For fixed $j$, as $t \searrow T$, we have:
\be
	\epsilon_{tj} \to 0    \;\; \text{and} \;\; \lambda_{tj} \to 0.
\ee	

Therefore for all $j$:
\be
	\lim_{t \searrow T} d_{\mathcal{F}}\left( (X, D(t)), (M, g(T) \right) \le   \vol_t (X \setminus U_2)
+ \vol_{g(T)}(M \setminus U_1),
\ee

Also since the diameter stays bounded as $t \searrow T$, one sees that as $j \to \infty$, 
\be\label{eq-vol-complement-2}
	\vol_{g(T)} (M \setminus U_1) \;\;  \text{and} \;\; \vol_t(X \setminus U_2) \to 0.
\ee

Therefore,
\be
	\lim_{t \searrow T} d_{\mathcal{F}} \Big( \left(X, D(t) \right) , \left( M, g(T)\right) \Big) = 0.
\ee

\end{proof}

\begin{rmrk}\label{remark-vol-convergence-2}
	Notice that since the diameter is bounded as $t \searrow T$, (\ref{eq-vol-complement-2}) gives
\be
	\vol \left( M_1 , g_1(t)\right) + \vol \left( M_2 , g_2(t) \right) \to \vol \left( M , g(T) \right)
\ee
\end{rmrk}

By applying Theorem \ref{thm-sub-diffeo-new}, we can also find an estimate on the Intrinsic Flat distance between two post surgery integral current spaces at times $0 < t_1 < t_2$. 

\begin{thm}
Suppose $\left(X , D(t) , T(t) \right)$ is as before, and the smooth flows $ \left( M_1, g_1(t) \right)$ and $\left( M_1, g_1(t) \right)$ do not encounter singularities on $(0,T)$, then as $t \to t_0 \in (0,T) $, we have
\be
	\lim_{t \to t_0} d_{\mathcal{F}} \Big( \left(X, D(t) \right) ,  \left( X, D(t) \right) \Big) = 0.
\ee
\end{thm}

\begin{proof}

Since the post surgery flows do not encounter any other singularity on $(0,T)$, we have
\be
	\sup_{M_i \times [t_0 - \delta , t_0 + \delta]} ||\Rm|| \le C = C(\delta)
\ee
with respect to a fixed background metric $g_0$ on $M_1 \sqcup M_2$. 

Let $\epsilon(t) =  \sqrt{e^{\sqrt{n} C |t - t_0|}} - 1 $ and let $\omega_j >0$ be a sequence with
\be
	\lim_{j \to \infty} \omega_j = 0.
\ee

Define $U_i$ as in Lemma \ref{lem-post-surgery}. Then, we have the following estimate on the metric distortion
\begin{eqnarray}
	\lambda_{tj} &=& \sup_{x,y \in U_1} \left| d_{(X , D(t_0))}(x,y) - d_{(X , D(t))}(x, y) \right| \notag \\ &\le& |L(t) - L(t_0)| + \Big((1+\epsilon(t))^2 - 1 \Big) \Big(\diam\left(M_1 , g_1(t) \right) + \diam \left( M_2 , g_2(t)\right) \Big)
\end{eqnarray}

We observe that as $t \to t_0$, $\epsilon(t) \to 0$ and $L(t) \to L(t_0)$ due to continuity therefore, $\lambda_{tj} \to 0$. The rest of the proof is the same as in Lemma \ref{lem-post-surgery}.

\end{proof}


\bibliographystyle{amsalpha}
\bibliography{reference2012}


\end{document}